\documentclass[a4paper, reqno]{amsart}

\usepackage[T1]{fontenc}
\usepackage{lmodern, microtype, amssymb}
\usepackage[scr=dutchcal]{mathalpha}
\usepackage[dvipsnames]{xcolor}
\usepackage[colorlinks=true, citecolor=RoyalBlue, linkcolor=black, urlcolor=RoyalBlue]{hyperref}
\usepackage{tikz-cd}

\microtypecontext{spacing=nonfrench}
\microtypesetup{stretch=20, shrink=15, spacing=true}
\tikzcdset{arrow style=tikz, diagrams={>={Computer Modern Rightarrow[length=5pt,width=4pt]}}}
\newtheorem{theorem}{Theorem}[section]
\newtheorem*{main-theorem}{Main theorem}
\newtheorem{proposition}[theorem]{Proposition}
\newtheorem{corollary}[theorem]{Corollary}
\newtheorem{lemma}[theorem]{Lemma}

\theoremstyle{definition}
\newtheorem*{lubin-conjecture}{Lubin's conjecture}

\newtheorem{definition}[theorem]{Definition}

\theoremstyle{remark}
\newtheorem{remark}[theorem]{Remark}

\DeclareMathOperator\bt{BT}
\DeclareMathOperator\BT{\mathbf{BT}}
\DeclareMathOperator\Comm{Comm}

\DeclareMathOperator\D{\mathbf{D}}
\DeclareMathOperator\Dsp{Dsp}
\DeclareMathOperator\End{End}
\DeclareMathOperator\FGL{\mathbf{FGL}}
\DeclareMathOperator\Fil{Fil}
\DeclareMathOperator\FIL{\mathbf{Fil}}
\DeclareMathOperator\Frac{Frac}
\DeclareMathOperator\Gal{Gal}
\DeclareMathOperator\Hom{Hom}
\DeclareMathOperator\Log{Log}

\DeclareMathOperator\MF{\mathbf{MF}}
\DeclareMathOperator\REP{\mathbf{Rep}}

\DeclareMathOperator\T{T}
\DeclareMathOperator\Aut{Aut}

\DeclareMathOperator\W{W}
\DeclareMathOperator\wideg{wideg}

\def\Aty{A_\infty}
\def\AinfL{A_{\mathrm{inf},L}}
\def\BinfL{B_{\mathrm{inf},L}}
\def\Acris{A_\mathrm{cris}}
\def\Bcris{B_\mathrm{cris}}
\def\BdR{B_\mathrm{dR}}

\usepackage{changepage}
\usepackage{datetime, inconsolata}

\title{Lubin's conjecture for height-one $p$-adic dynamical systems}
\author{Martin Debaisieux}
\address{Département de Mathématique, Université de Mons, 7000 Mons, Belgium}
\email{martin.debaisieux@umons.ac.be}
\thanks{The author was supported by the FRIA of the Fonds de la Recherche Scientifique -- FNRS}
\subjclass{\texttt{11S82}, \texttt{11S31}, \texttt{37P20}; \texttt{14L05}, \texttt{11F80}, \texttt{11F85}}
\date{\today}

\begin{document}

\begin{adjustwidth}{2mm}{2mm}
    \begin{abstract}
        We prove Lubin's conjecture for height-one commuting pairs of formal power series defined over the ring of integers $\mathcal{O}$ of an arbitrary finite extension of $\mathbb{Q}_p$. Using $p$-adic Hodge theory, we show that such a pair is a pair of endomorphisms of a height-one formal group defined over $\mathcal{O}$. Our constructions are functorial.
    \end{abstract}
\end{adjustwidth}

\maketitle

\vspace{-3mm}

\section{Introduction}

\noindent In \cite{Lub94}, Lubin studied pairs of commuting formal power series defined over the ring of integers of a finite extension of the field of $p$-adic numbers $\mathbb{Q}_p\!$ with a fixed point at $0$. Such pairs give rise to dynamical systems on the $p$-adic open unit disk and are known as $p$-adic dynamical systems. Lubin noticed that ‘‘for an invertible series to commute with a noninvertible series, there must be a formal group somehow in the background''. Various results in this direction were proved in the subsequent years, see for instance \cite{Li96a}, \cite{Li96b}, \cite{Li02a}, \cite{Li02b}, \cite{Sar05}, and the conjecture was formulated more precisely to what we call Lubin's conjecture.

\begin{lubin-conjecture}[{\cite[p. 131]{Sar05}}]
    Suppose that $f$ and $u$ are a noninvertible and a nontorsion invertible series, respectively, defined over the ring of integers~$\mathcal{O}$ of a finite extension of $\mathbb{Q}_p$. Suppose further that the roots of $f$ and all of its iterates are simple, and that $f'(0)$ is a uniformizer in $\mathbb{Z}_p$. If $f \circ u = u \circ f$, then $f, u \in \End_\mathcal{O}(F)$ for some formal group $F$ over $\mathcal{O}$.
\end{lubin-conjecture}

    \subsection{Progression} Under the hypotheses of Lubin's conjecture and assuming that the Weierstrass degree $\wideg(f)$ of $f$ is finite, Lubin proved in \cite[Main Theorem 6.3]{Lub94} that $\wideg(f)$ is of the form $p^h$ for an integer $h \geqslant 1$ called the height of $(f, u)$. If Lubin's conjecture holds, $h$ is the height of the formal group. The conjecture was first solved in the case of height $1$ over $\mathbb{Z}_p$ by Specter in \cite{Spe18}, then for all finite heights over $\mathbb{Z}_p$ by Berger in \cite{Ber19}. We solve in \cite{De26a} the cases of height $1$ over extensions of $\mathbb{Z}_p$ whose ramification index is coprime to $p(p-1)$, provided that the linear coefficient of the invertible power series $u$ lies in $\mathbb{Z}_p$. The approach is divided into the three following main steps:
    \begin{enumerate}
        \item Studying the dynamical systems induced by $f$ and $u$ in order to extract a Galois character from the set of $f$-consistent sequences.
        \item Showing that this character is crystalline of Hodge-Tate weight $1$ and use it to define a $\mathbb{Z}_p$-module structure on the Galois set of $f$-consistent sequences for which $f$ is an endomorphism.
        \item Applying explicit functors from integral $p$-adic Hodge theory to recover the latent formal group.
    \end{enumerate}
    This can be seen as a generalization of \cite{Spe18}, where Specter retrieved the underlying crystalline character which, in turn, determines the Lubin-Tate endomorphism. However, the formal groups arising in this more general setting may not be Lubin-Tate (they are if and only if the base ring is $\mathbb{Z}_p$), and so the reconstruction of the crystalline character alone is not sufficient to conclude, leading us to the subsequent steps. Only (1) depends on the ramification assumption, while the constructions in (2) and (3) are independent of it. In this article, we remove this assumption, and prove the height-one cases over any finite extension. We introduce a new argument based on a result of \cite{Li96a} to replace the one used in (1), allowing us to repeat the remainder of the constructions. Additionally, we show that these constructions are functorial. The other cases of Lubin's conjecture are still open.

    \subsection{Notations and main statement} Let $p$ be a prime number. We fix $\overline{\mathbb{Q}}_p$ an algebraic closure of $\mathbb{Q}_p$ containing all the considered extensions and denote by $\mathbb{C}_p$ its $p$-adic completion with open unit disk $\mathfrak{m}_{\mathbb{C}_p}\!$. If $E$ is a finite extension of $\mathbb{Q}_p$, let $\mathcal{O}_E$ be its ring of integers with maximal ideal $\mathfrak{m}_E$, and write $G_E = \Gal(\overline{\mathbb{Q}}_p/E)$ for its absolute Galois group. Given a ring $R$, we define $\mathscr{S}_0(R) = TR[\![T]\!]$ to be the set of formal power series over $R$ without constant term. Composition induces a law on $\mathscr{S}_0(R)$ making it into a noncommutative monoid, and a necessary and sufficient condition on $s \in \mathscr{S}_0(R)$ to be invertible is that $s'(0)$ is a unit of $R$.

    \begin{main-theorem}[\ref{main-theorem}]
        Let $K$ be a finite extension of $\mathbb{Q}_p$. Let $(f, u)$ be a commuting pair of elements in $\mathscr{S}_0(\mathcal{O}_K)\!$ such that:
        \begin{enumerate}
            \item[$\triangleright$] $f'(0)$ is a uniformizer in $\mathbb{Z}_p$, the roots of $f$ and all of its iterates in $\mathfrak{m}_{\mathbb{C}_p}\!$ are simple and $\wideg(f) = p$;
            \item[$\triangleright$] $u'(0) \in \mathbb{Z}_p^\times$ is nontorsion.
        \end{enumerate}
        Then there exists a formal group $F$ defined over $\mathcal{O}_K$ such that $f, u \in \End_{\mathcal{O}_K}(F)$.
    \end{main-theorem}

    By a formal group, we mean a one-dimensional commutative formal group law; see \cite{Haz12} for background. As in \cite{De26a}, we pick $u'(0)$ in $\mathbb{Z}_p^\times$ rather than in $\mathcal{O}_K^\times$, which allows us to define a character with values in the $p$-adic integers. This choice is reminiscent of the following fact: the linear coefficient of an endomorphism of a height-one formal group over $\mathcal{O}_K$ lies in $\mathbb{Z}_p$ (\cite[35.3.18]{Haz12}).

\section{Strategy}

\noindent In this section, we expose the motivation behind the procedure used to solve Lubin's conjecture in height $1$. This motivation comes from well-known results regarding height-one formal groups.

    \subsection{Motivation}\label{2.1. Motivation} Let $K$ be a finite extension of $\mathbb{Q}_p$. A height-one formal group $F$ defined over $\mathcal{O}_K$ has a logarithm $\Log_F \in \mathscr{S}_0(K)$ which is the unique isomorphism over $K$ from $F$ to the additive formal group such that $\Log_F'(0) = 1$. Over $K$, one thus has that
    \[
        F(X,Y) = \Log_F^{\circ -1}(\Log_F(X) + \Log_F(Y))
    \]
    but this is not sufficient, in general, to determine the integrality of the coefficients of $F$. However, its integral $p$-adic Galois representation $\T_p(F)$ characterizes $F$ up to isomorphism. Indeed, the category of height-one formal groups defined over $\mathcal{O}_K$ is equivalent to the category of crystalline characters of $G_K$ of weight $1$ under the functor $\T_p$ (see \S\ref{3.3.2. Crystalline representations} for a more precise statement). Morphisms are the same, and there are explicit functors to pass from one category to the other (see \S\ref{3.3. Objects arising from formal groups}). Recall that the $p$-adic Tate module of $F$ is defined by
    \[
        \T_p(F) = \varinjlim_{x \mapsto [p]_F(x)} F[p^n]
    \]
    where $F[p^n]$ is the group of $p^n$-torsion points of $F$ in $\mathfrak{m}_{\mathbb{C}_p\!}$ for $n \geqslant 0$, and $f = [p]_F$ is the multiplication-by-$p$ endomorphism in $F$. The Tate module of $F$ has a dynamical flavor since it arises as the set of $f$-consistent sequences endowed with the structure arising from $F$. The absolute Galois group $G_K$ acts on $\T_p(F)$ through a continuous character
    \[
        \chi_F \colon G_K \longrightarrow \mathbb{Z}_p^\times.
    \]
    According to \cite[Théorème 5]{Ser67}, the image of $\chi_F$ is open, and therefore contains $1+p^n\mathbb{Z}_p$ for all $n$ sufficiently large. On the other side, $\End_{\mathcal{O}_K}(F) \simeq \mathbb{Z}_p$ via $s \mapsto s'(0)$ in height $1$. Given a nontorsion $u \in \Aut_{\mathcal{O}_K}(F)$, we may assume, upon replacing $u$ by one of its $\mathbb{Z}_p$-iterates that $\upsilon_p(u'(0) - 1) = n$ for a sufficiently large $n$. These give rise to the following commutative diagram:
    \[
        \begin{tikzcd}[column sep = 11mm, row sep = 13mm]
            G_K \ar[r, twoheadrightarrow, "\chi_F"] & \chi_F(G_K) \ar[r, hookrightarrow] & \mathbb{Z}_p^\times \ar[r, "\sim"] & \Aut_{\mathcal{O}_K}(F)\\
            G_L \ar[u, hookrightarrow]\ar[r, twoheadrightarrow] & 1+p^n\mathbb{Z}_p \ar[u, hookrightarrow]\ar[rr, "\sim"] & & u^{\circ \mathbb{Z}_p} \ar[u, hookrightarrow]
        \end{tikzcd}
    \]
    where $L/K$ is the finite Galois extension corresponding to the clopen $\chi_F^{-1}(1+p^n \mathbb{Z}_p)$ by Galois theory. Therefore the $\mathbb{Z}_p$-iterates of $u$ furnish a restricted $G_L$-action on $\T_p(F)$, which is the Tate module of the scalar extension of $F$ to $\mathcal{O}_L$. The previous functors then allow to recover $F$ from its $p$-adic Tate module, while the descent to the ring $\mathcal{O}_K$ is given by the logarithm $\Log_F$ which uniquely determines the formal group $F$ over $K$.

    \subsection{Procedure} The previous paragraph suggests that Lubin's conjecture holds in height $1$: let $(f, u)$ be a commuting pair defined over $\mathcal{O}_K$ as in the statement of Main Theorem. Then the noninvertible power series $f$ can be used to retrieve the underlying set $T_f$ of the latent $p$-adic Tate module, while the invertible one $u$ can be used to describe a part of the Galois action on the set $T_f$ which, if crystalline of weight $1$, carries the data of a formal group.
    \begin{enumerate}
        \item[$\triangleright$] In Section \ref{3. Theoretical background}, we recall the necessary background in (integral) $p$-adic Hodge theory for the next section. In \S\ref{3.1. Rings for p-adic Hodge theory}, we briefly expose the construction of some rings that will be used later. In \S\ref{3.2. Galois p-adic representations}, we recall the definition of being crystalline of weight $1$ and of the associated functors. In \S\ref{3.3. Objects arising from formal groups}, we survey the objects attached to a formal group that will allow us to recover the latent formal group from its representation.

        \item[$\triangleright$] In Section \ref{4. Constructions}, we carry out the technical constructions from $(f, u)$ to recover its formal group over $\mathcal{O}_K$ as in \cite{De26a} but in the more general setting of an arbitrary finite extension $K$ of $\mathbb{Q}_p$. In \S\ref{4.1. The dynamical system}, we study the logarithm and the torsion points of $(f, u)$. This study allows us to define, in \S\ref{4.2. The latent Tate module}, a Galois character which is associated with $T_f$ and is crystalline of weight $1$. This representation is then used to retrieve the missing $\mathbb{Z}_p$-module structure on the Galois set $T_f$ for which $f$ is an endomorphism. In \S\ref{4.3. The associated Breuil-Kisin module}--\S\ref{4.6. The associated formal group}, we follow the aforementioned path to recover a formal group from $(T_f, f)$. At each step, we keep track of the behavior of $f$ on the constructed object to ensure that the formal group obtained admits $f$ as an endomorphism. As a result it also admits $u$.

        \item[$\triangleright$] In Section \ref{5. Functoriality}, we show that the previous constructions are functorial. They define an equivalence of categories between height-one Lubin pairs over $\mathcal{O}_K$ and height-one formal groups over $\mathcal{O}_K$. In \S\ref{5.1. Height-one Lubin pairs}, we define the category of height-one Lubin pairs for which the considered $p$-adic dynamical systems in this paper are objects. In \S\ref{5.2. Equivalence of categories}, we prove the latter equivalence. In \S\ref{SS: full-ht1}, we apply this functor to deduce the main result of \cite{De26b}.
    \end{enumerate}

\section{Theoretical background}\label{3. Theoretical background}

\noindent Let $L$ be a finite extension of $\mathbb{Q}_p$ (in Section \ref{4. Constructions}, the field $L$ will be specialized to the finite extension of the base field $K$ generated by the fixed points of the invertible series $u$). Let $\pi_L$ be a uniformizer of $\mathcal{O}_L$ and let $q$ be the cardinality of the residue field $\mathcal{O}_L/\mathfrak{m}_L = \mathbb{F}_q$. Denote by $L_0$ the maximal unramified extension of $\mathbb{Q}_p$ inside $L$ with Frobenius endomorphism $\sigma$, and by $E(T)$ the Eisenstein minimal polynomial of the uniformizer $\pi_L$ over $L_0$.

    \subsection{Rings for \texorpdfstring{$p$}{p}-adic Hodge theory}\label{3.1. Rings for p-adic Hodge theory} In this subsection, we recall the construction of some of the rings in \cite{Fon94}, including the period rings $\BdR$ and $\Bcris$, which are relevant to the study of representations arising from formal groups. We then turn to the rings $\mathfrak{S}$ and $S$ from \cite{Kis06}, which provide a framework for studying integral representations arising from formal groups.

        \subsubsection{The integers of the tilt of \texorpdfstring{$\mathbb{C}_p$}{Cp}} Denote by $\mathcal{O}_{\mathbb{C}_p\!}$ the closed unit disk of $\mathbb{C}_p$. The starting point of the definition of Fontaine's period rings is the perfect ring
        \[
            \mathcal{R} = \varinjlim_{x \mapsto x^q} \mathcal{O}_{\mathbb{C}_p}/(\pi_L).
        \]
        of characteristic $p$. This ring is endowed with a valuation $\upsilon_\mathcal{R}$ for which it is complete, defined by $\upsilon_\mathcal{R}(x) = \lim_{n \mapsto +\infty} q^n \upsilon_p(\widehat{x}_n)$ for all $x = (x_n)_{n \geqslant 0} \in \mathcal{R}$, where $\widehat{x}_n \in \mathcal{O}_{\mathbb{C}_p\!}$ is a lift of $x_n$ modulo $(\pi_L)$ for all integers $n \geqslant 0$ and where $\upsilon_p$ is the $p$-adic valuation normalized on $\mathbb{Q}_p$. This ring is also equipped with a compatible action of $G_L$ arising from the $G_L$-action on $\mathcal{O}_{\mathbb{C}_p\!}$.

        \subsubsection{Blank period rings} Let
        \[
            \AinfL = \W_{\!\pi_L\!}(\mathcal{R})
        \]
        be the ring of $\pi_L$-adic Witt vectors of $\mathcal{R}$. It is the unique strict $\pi_L$-ring of residue ring $\mathcal{R}$. Alternatively, one can see this blank period ring as $\mathcal{O}_L \otimes_{\W(\mathbb{F}_q)} \W(\mathcal{R})$ with $\W$ the usual $p$-adic Witt vectors functor. Its elements are uniquely written in the form $\sum_{n \geqslant 0} [x_n] \pi_L^n$, where $x_n \in \mathcal{R}$ and $[-]$ is the Teichmüller lift. The action of $G_L$ on $\mathcal{R}$ lifts by functoriality to a $G_L$-action on $\AinfL$, and there is a $G_L$-equivariant epimorphism $\theta \colon \AinfL \twoheadrightarrow \mathcal{O}_{\mathbb{C}_p\!}$ of $\mathcal{O}_L$-algebras determined by the requirement that $[x]$ is sent to $\lim_{n \to +\infty} \widehat{x}_n^{q^n}$, where $\widehat{x}_n \in \mathcal{O}_{\mathbb{C}_p\!}$ lifts the $n$-th component of $x \in \mathcal{R}$. This map uniquely extends to a $G_L$-equivariant epimorphism $\theta$ from $\BinfL^+ = \AinfL[1/\pi_L]$ to $\mathbb{C}_p$ by $L$-linearity and its kernel is principal.

        \subsubsection{The de Rham period field} The de Rham period ring
        \[
            \BdR^+ = \varprojlim_{n \geqslant 1} \BinfL^+/\ker(\theta)^n
        \]
        is the $\ker(\theta)$-adic completion of $\BinfL^+$. It inherits the action of $G_L$ on $\BinfL^+$ which is also compatible with the $\ker(\theta)$-adic topology. Once again, the map $\theta$ extends to a $G_L$-equivariant epimorphism $\theta \colon \BdR^+ \twoheadrightarrow \mathbb{C}_p$ and its kernel remains principal. The $\mathbb{Z}$-powers of the ideal $\ker(\theta)$ induce on the fraction field $\BdR$ of $\BdR^+$ a decreasing, exhaustive and separated filtration $(\Fil^n \BdR)_{n \in \mathbb{Z}}$. The elements of $\BdR$ fixed under the action of $G_L$ are precisely the elements of $L$.

        \subsubsection{The crystalline period ring} Define $\Acris$ to be the $p$-adic completion of the PD-envelope of $\W(\mathcal{R})$ with respect to $\ker(\theta)$. This remains a subring of $\BdR$. We then set
        \[
            \Bcris^+ = \Acris[1/p].
        \]
        The action of $G_L$ on $\W(\mathcal{R})$ extends at each previous ring and the latter contains a unique $G_L$-stable $\mathbb{Q}_p$-line on which $G_L$ acts by the cyclotomic character. Then the crystalline period ring $\Bcris = \Bcris^+[1/t]$ is obtained by inverting any nonzero $t$ in this line. The Frobenius endomorphism $\varphi$ on $\W(\mathcal{R})$ extends uniquely to it, where $\varphi(t) = pt$, into an injective $G_L$-equivariant endomorphism $\varphi \colon \Bcris \hookrightarrow \Bcris$ which is semilinear with respect to $\sigma$. The elements of $\Bcris$ fixed under the action of $G_L$ are the elements of $L_0$. Moreover, the canonical map $L \otimes_{L_0} \Bcris \rightarrow \BdR$ is injective and $G_L$-equivariant.

        \subsubsection{The Kisin ring} Let
        \[
            \mathfrak{S} = \W(\mathbb{F}_q)[\![T]\!]
        \]
        be the Kisin ring. This ring is equipped with a semilinear Frobenius endomorphism $\varphi$ with respect to $\sigma$, given by $\varphi(T) = T^p$. There exists an embedding $\mathfrak{S} \hookrightarrow \W(\mathcal{R})$ preserving the Frobenius (\cite[\S2.1.1]{Kis06}). Let $\mathscr{E}$ be the fraction field of the $p$-adic completion of $\mathfrak{S}[1/T]$. This embedding extends to an embedding $\mathscr{E} \hookrightarrow \W(\Frac (\mathcal{R}))$, and let $\mathscr{E}^{\mathrm{ur}}$ be the maximal unramified extension of $\mathscr{E}$ inside $\W(\Frac (\mathcal{R}))[1/p]$ and let $\mathcal{O}_{\mathscr{E}^\mathrm{ur}}$ be its ring of integers. We define
        \[
            \mathfrak{S}^\mathrm{ur} = \mathcal{O}_{\mathscr{E}^\mathrm{ur}} \cap \W(\mathcal{R}).
        \]
        The absolute Galois group of $L_\infty$, the extension generated over $L$ by a compatible sequence of $p$-th roots of $\pi_L$ in $\mathfrak{m}_{\mathbb{C}_p\!}$, fixes $\mathfrak{S}$ and therefore acts naturally on $\mathfrak{S}^\mathrm{ur}$.

        \subsubsection{The Breuil ring} Denote by $S$ the $p$-adic completion of the PD-envelope of $\W(\mathbb{F}_q)[T]$ with respect to the ideal $(E(T))$, and let $\Fil^1 S \subset S$ be the closure of the ideal generated by the divided powers of $E(T)$. As topological rings, $S/\Fil^1 S \simeq \mathcal{O}_L$. By definition $\mathfrak{S} \subset S$, and the Frobenius $\varphi$ on $\mathfrak{S}$ extends to a $\sigma$-semilinear Frobenius endomorphism $\varphi$ on the Breuil ring $S$.

    \subsection{Galois \texorpdfstring{$p$}{p}-adic representations}\label{3.2. Galois p-adic representations} Fontaine has classified the $p$-adic representations of $G_L$. We expose here two classes that are relevant to our problem. Recall that a $p$-adic representation $(V, \rho)$ of $G_L$ is a finite-dimensional $\mathbb{Q}_p$-vector space $V$ with a continuous linear action $\rho$ of $G_L$. One-dimensional representations are given by continuous characters $\rho \colon G_L \rightarrow \mathbb{Z}_p^\times$ by compacity of $G_L$. Denote by $\REP_{\mathbb{Q}_p}(G_L)$ the category of $p$-adic representations of $G_L$.

        \subsubsection{de Rham representations} Denote by $\FIL_L$ the category of finite-dimensional $L$-vector spaces equipped with a decreasing, exhaustive and separated $\mathbb{Z}$-filtration by subspaces. Fontaine defined a contravariant functor
        \[
            \D_\mathrm{dR}^* \colon \REP_{\mathbb{Q}_p}(G_L) \longrightarrow \FIL_L
        \]
        by sending a $p$-adic representation $V$ to $\D_\mathrm{dR}^*(V) = \Hom_{\mathbb{Q}_p[G_L]}(V, \BdR)$, equipped with the induced filtration $\Fil^n\D_\mathrm{dR}^*(V) = \Hom_{\mathbb{Q}_p[G_L]}(V, \Fil^n\BdR)$ for all $n \in \mathbb{Z}$. A $p$-adic representation $V$ of $G_L$ is said to be de Rham if the $L$-dimension of $\D_\mathrm{dR}^*(V)$ is equal to the $\mathbb{Q}_p$-dimension of $V$. In this case, the jumps of the filtration are the Hodge-Tate weights of $V$. When $(V, \rho)$ is one-dimensional, being de Rham amounts to require the existence of a nonzero period  $z \in \BdR$ such that $g.z = \rho(g)z$ for all $g \in G_L$. If so, its weight is the largest $w \in \mathbb{Z}$ such that $z \in \Fil^w \BdR$.

        \subsubsection{Crystalline representations} Let $\MF_L(\varphi)$ be the category of finite-dimensional vector spaces over $L_0$ with a bijective $\sigma$-semilinear Frobenius endomorphism $\varphi$ and whose scalar extension to $L$ are in $\FIL_L$. Morphisms are $L_0$-linear maps commuting with $\varphi$ and preserving the filtration after scalar extension. Again, Fontaine defined a contravariant functor
        \[
            \D_\mathrm{cris}^* \colon \REP_{\mathbb{Q}_p}(G_L) \longrightarrow \MF_L(\varphi)
        \]
        by sending a $p$-adic representation $V$ to $\D_\mathrm{cris}^*(V) = \Hom_{\mathbb{Q}_p[G_L]}(V, \Bcris)$, equipped with the Frobenius induced by that on $\Bcris$, and its filtration over $L$ arises from the embedding $L \otimes_{L_0} \D_\mathrm{cris}^*(V) \hookrightarrow \D_\mathrm{dR}^*(V)$. A $p$-adic representation $V$ of $G_L$ is said to be crystalline if the $L_0$-dimension of $\D_\mathrm{cris}^*(V)$ is equal to the $\mathbb{Q}_p$-dimension of~$V$. Then its Hodge-Tate weights are its weights as a de Rham representation. One can show that the restriction of the previous functor to crystalline representations is an anti-equivalence of categories
        \[
            \D_\mathrm{cris}^* \colon \REP_{\mathbb{Q}_p}^\mathrm{cris}(G_L) \longrightarrow \MF_L^\mathrm{adm}(\varphi)
        \]
        where the superscript “cris'' stands for the crystalline ones, and “adm'' stands for the admissible ones. Concretely, a one-dimensional filtered $\varphi$-module $D$ with jump at $j \in \mathbb{Z}$ is admissible if and only if $j = \upsilon_p(\varphi(d)/d)$ for any nonzero $d \in D$.

    \subsection{Objects arising from formal groups}\label{3.3. Objects arising from formal groups} Before recovering the latent formal group attached to our dynamical system, we will first recover its latent $p$-adic Tate module, which characterizes the formal group up to isomorphism. Here, we make precise the latter statement and introduce the categories and functors to pass from one to the other. The path followed in Section \ref{4. Constructions}, once the latent Tate module has been recovered, is indicated by dashed arrows in Figure \ref{Figure 1}. We let $\FGL\!/\mathcal{O}_L$ be the category of formal group laws of finite height and dimension over $\mathcal{O}_L$.

    \begin{figure}[h]
        \centering
        \begin{tikzcd}[row sep=11mm, column sep=10mm]
            &&  \REP_{\mathbb{Q}_p}^{\mathrm{cris},\{0,1\}}(G_L)
                    \ar[r, shift left=.6ex, "\mathbf{D}_\mathrm{cris}^*"]
            &   \MF_L^{\mathrm{adm},\{0,1\}}(\varphi)
                    \ar[l, shift left=.6ex]
            \\
            &   p\textbf{-}\mathbf{Div}/\mathcal{O}_L
                    \ar[r, shift left=.6ex, "\T_p"]
            &   \REP_{\mathbb{Z}_p}^{\mathrm{cris},\{0,1\}}(G_L)
                    \ar[r, shift left=.6ex, "\mathbf{D}_\mathfrak{S}^*"]
                    \ar[l, shift left=.6ex]
                    \ar[u, hookrightarrow, "\mathbb{Q}_p \otimes_{\mathbb{Z}_p} -"]
            &   \BT_\mathfrak{S}(\varphi)
                    \ar[l, shift left=.6ex]
                    \ar[u, hookrightarrow, "D"]
            \\
                \FGL/\mathcal{O}_L
                    \ar[r, shift left=.6ex]
            &   p\textbf{-}\mathbf{Div}^\mathrm{c}\!/\mathcal{O}_L
                    \ar[r, shift left=.6ex, "\T_p"]
                    \ar[l, shift left=.6ex, dashed, "\text{comult.}"]
                    \ar[dr, shift left=.6ex, bend right=25]
                    \ar[u, hookrightarrow]
            &   \REP_{\mathbb{Z}_p}^{\mathrm{FGL}}(G_L)
                    \ar[r, shift left=.6ex, dashed, "\mathbf{D}_\mathfrak{S}^*"]
                    \ar[l, shift left=.6ex]
                    \ar[u, hookrightarrow]
            &   \BT^{\mathrm{f}}_\mathfrak{S}(\varphi)
                    \ar[l, shift left=.6ex]
                    \ar[dl, shift left=.6ex, dashed, "S \otimes_{\varphi, \mathfrak{S}} -", bend left=25]
                    \ar[u, hookrightarrow]
            \\
            &&  S\textbf{-}\mathbf{Win}/\mathcal{O}_L
                    \ar[ur, shift left=.6ex, bend right=25]
                    \ar[ul, shift left=.6ex, dashed, "\text{BT} \circ \text{Dsp}", bend left=25]
        \end{tikzcd}
        \caption{Categories arising from formal groups.}
        \label{Figure 1}
    \end{figure}

        \subsubsection{\texorpdfstring{$p$}{p}-divisible groups} A $p$-divisible group of height $h \geqslant 0$ defined over $\mathcal{O}_L$ is an inductive system $\Gamma = \{(\Gamma_n, i_n) \;;\; n \geqslant 0\}$ of finite affine group schemes $\Gamma_n$ over $\mathcal{O}_L$ of order $p^{nh}$ such that $i_n \colon \Gamma_n \rightarrow \Gamma_{n+1}$ identifies $\Gamma_n$ with the kernel of multiplication by $p^n$ in $\Gamma_{n+1}$ for all integers $n \geqslant 0$. Morphisms are systems of morphisms of affine group schemes over $\mathcal{O}_L$ that are compatible with transition maps. Let $p\text{-}\mathbf{Div}/\mathcal{O}_L$ be the category of $p$-divisible groups over $\mathcal{O}_L$. Taking the comultiplication of the Hopf algebra yields an equivalence between the full subcategory $p\textbf{-}\mathbf{Div}^\mathrm{c}\!/\mathcal{O}_L$ of connected objects and the category $\mathbf{FGL}/\mathcal{O}_L$ (see \cite[Proposition 1]{Tat67}).

        \subsubsection{Crystalline representations}\label{3.3.2. Crystalline representations} Breuil in \cite[Theorem 5.3.2]{Bre00} for $p \neq 2$ and then Kisin in \cite[Corollary 2.2.6]{Kis06} for any $p$ showed that the $p$-adic Tate module functor
        \[
            \T_p \colon p\textbf{-}\mathbf{Div}/\mathcal{O}_L \longrightarrow \REP_{\mathbb{Z}_p}^{\mathrm{cris}, \{0,1\}}
        \]
        is an equivalence, where the RHS category is the one of $G_L$-stable $\mathbb{Z}_p$-lattices inside crystalline representations of $G_L$ of Hodge-Tate weights in $\{0,1\}$. The height then corresponds to the rank of the lattice. Let $\REP_{\mathbb{Z}_p}^{\mathrm{FGL}}(G_L)$ be the essential image of the full subcategory of the LHS one of connected objects under the functor $\T_p$. This equivalence associate height-one formal groups over $\mathcal{O}_L$ with $G_L$-stable $\mathbb{Z}_p$-lattices of one-dimensional crystalline $G_L$-representations of weight $1$.

        \subsubsection{Breuil-Kisin modules} A Breuil-Kisin module $(\mathfrak{M}, \varphi)$ is a finite free $\mathfrak{S}$-module $\mathfrak{M}$ equipped with a $\varphi_\mathfrak{S}$-semilinear Frobenius endomorphism $\varphi \colon \mathfrak{M} \rightarrow \mathfrak{M}$ whose cokernel of its $\mathfrak{S}$-linearization $1 \otimes \varphi \colon \mathfrak{S} \otimes_{\mathfrak{S}, \varphi} \mathfrak{M} \rightarrow \mathfrak{M}$ is killed by $E(T)$. Together with $\varphi$-equivariant $\mathfrak{S}$-linear maps, these objects form the category $\mathbf{BT}_\mathfrak{S}(\varphi)$. Kisin obtained in \cite{Kis06} a functor $\mathbf{D}_\mathfrak{S}^*$ that completes the following diagram and makes it commutative:
        \[
            \begin{tikzcd}[row sep=11mm, column sep=15mm]
                    \mathbf{Rep}_{\mathbb{Q}_p}^{\mathrm{cris}, \{0,1\}}
                        \ar[r, shift left=.6ex, "\mathbf{D}_\mathrm{cris}^*"]
                &   \mathbf{MF}_L^{\mathrm{adm}, \{0,1\}}(\varphi)
                        \ar[l, shift left=.6ex]
                \\
                    \mathbf{Rep}_{\mathbb{Z}_p}^{\mathrm{cris}, \{0,1\}}
                        \ar[r, "\mathbf{D}_\mathfrak{S}^*"]
                        \ar[u, hookrightarrow, "\mathbb{Q}_p \otimes_{\mathbb{Z}_p} -"]
                &   \mathbf{BT}_\mathfrak{S}(\varphi)
                        \ar[u, hookrightarrow, "D"]
            \end{tikzcd}
        \]
        where $D$ is the functor defined in \cite[\S1]{Kis06}, while ‘‘$\{0,1\}$'' stands for the weights or jumps in $\{0,1\}$. The functor $\mathbf{D}_\mathfrak{S}^*$ is an anti-equivalence if $p$ is odd. In any case, it induces an anti-equivalence between representations arising from formal group laws and its essential image $\mathbf{BT}_\mathfrak{S}^\mathrm{f}(\varphi)$ of formal Breuil-Kisin modules.

        \subsubsection{\texorpdfstring{$S$}{S}-windows} Let $R \in \{\mathcal{O}_L/(p^n)\,;\, n \geqslant 0\}$ be a ring. An $S$-window $(M, M', \varphi)$ over $R$ is a finite projective $S$-module $M$ together with a submodule $M'$ containing $(\Fil^1 S)M$ and a $\sigma$-semilinear Frobenius endomorphism $\varphi \colon M\rightarrow M$ such that:
        \begin{enumerate}
            \item $M/M'$ is a projective $R$-module,
            \item $\varphi(M') \subseteq pM$ and $M$ is generated by $\varphi(M) + \varphi/p(M')$ as an $S$-module,
            \item since $pM \subseteq (1 \otimes \varphi)(\varphi^*(M))$, where $\varphi^*(M) = S \otimes_{S, \varphi} M$, the map $1 \otimes \varphi$ is injective and there exists a unique $S$-linear map $\psi \colon M \rightarrow \varphi^*(M)$ such that $(1 \otimes \varphi) \circ \psi \colon M \rightarrow M$ is the multiplication-by-$p$ map. Then, for $n$ sufficiently large,
            \[
                \varphi^{n-1*}(\psi) \circ \varphi^{n-2*}(\psi) \circ \cdots \circ \psi \colon M \longrightarrow \varphi^{n*}(M)
            \]
            factors through $(p, \Fil^1 S)\varphi^{n*}(M)$.
        \end{enumerate}
        Together with $\varphi$-equivariant $S$-linear maps preserving the distinguished submodule, these objects form the category $S\textbf{-}\mathbf{Win}/R$. The extension
        \[
            S \otimes_{\varphi, \mathfrak{S}} - \colon \mathbf{BT}_\mathfrak{S}^\mathrm{f}(\varphi) \longrightarrow S\textbf{-}\mathbf{Win}/\mathcal{O}_L
        \]
        defined in \cite[\S1.2.3]{Kis09} is an equivalence of categories. In the particular case of $R = \mathcal{O}_L$, the data of an $S$-window defined over $\mathcal{O}_L$ is equivalent to the data of a system $(M_n)_{n \geqslant 1}$, where $M_n$ is an $S$-window defined over $\mathcal{O}_L/(p^n)$ and is obtained from $M_{n+1}$ by base change for every integer $n \geqslant 1$. Zink introduced in \cite{Zin01} two functors, namely $\Dsp$ and $\bt$, such that their composite makes the diagram
        \[
            \begin{tikzcd}[row sep=11mm, column sep=15mm]
                    p\textbf{-}\mathbf{Div}^\mathrm{c}\!/\mathcal{O}_L
                        \ar[r, shift left=.6ex, "\T_p"]
                        \ar[d, shift left=.6ex]
                &   \REP_{\mathbb{Z}_p}^{\mathrm{FGL}}(G_L)
                        \ar[l, shift left=.6ex]
                        \ar[d, shift left=.6ex, "\mathbf{D}_\mathfrak{S}^*"]
                \\
                    S\textbf{-}\mathbf{Win}/\mathcal{O}_L
                        \ar[r, shift left=.6ex]
                        \ar[u, shift left=.6ex, "\bt \circ \Dsp"]
                &   \BT_\mathfrak{S}^\mathrm{f}(\varphi)
                        \ar[u, shift left=.6ex]
                        \ar[l, shift left=.6ex, "S \otimes_{\varphi, \mathfrak{S}} -"]
            \end{tikzcd}
        \]
        commutative. This composite is defined on $S$-windows over $\mathcal{O}_L$ expressed as limits. The last equivalence completes the two bottom rows of the commutative diagram in Figure \ref{Figure 1} and provides the path we follow in the next section.

\section{Constructions}\label{4. Constructions}

\noindent Let $K$ be a finite extension of $\mathbb{Q}_p$. Throughout this article, we assume $(f, u)$ to be a commuting pair of formal power series in $\mathscr{S}_0(\mathcal{O}_K)$ with $f'(0)$ a uniformizer in $\mathbb{Z}_p$ and $u'(0) \in \mathbb{Z}_p^\times$ nontorsion, and such that the roots of $f$ and its iterates in $\mathfrak{m}_{\mathbb{C}_p\!}$ are simple and $\wideg(f) = p$. Recall that the Weierstrass degree of $s \in \mathscr{S}_0(\mathcal{O}_K)$ is the $T$-adic valuation of the power series $s(T) \bmod \mathfrak{m}_K$. By the Weierstrass preparation theorem, when $\wideg(s) = d$ is finite, $d$ is the number of roots of $s$ in $\mathfrak{m}_{\mathbb{C}_p\!}$ counting with multiplicities.

    \subsection{The dynamical system}\label{4.1. The dynamical system} We begin this section with a preliminary analysis of some data at hand in the dynamical system $(f, u)$, including the computation of the valuation of their roots and of their periodic points. These data play a crucial role in the recovery of the structure of the latent $p$-adic Tate module in \S\ref{4.2. The latent Tate module}. When we mention the roots or the periodic points of a power series, we mean its roots or its periodic points in $\mathfrak{m}_{\mathbb{C}_p\!}$.

        \subsubsection{The logarithm}\label{4.1.1. The logarithm} Lubin proved in \cite[\S1]{Lub94} that there is a unique invertible formal power series
        \[
            \Log_f(T) \in T + T^2K[\![T]\!]
        \]
        satisfying $\Log_f(s(T)) = s'(0) \Log_f(T)$ for all power series $s \in \mathscr{S}_0(\mathcal{O}_K)$ commuting with $f$, called the logarithm of $(f, u)$. This series converges on $\mathfrak{m}_{\mathbb{C}_p\!}$, and its derivative has coefficients in $\mathcal{O}_K$ because the roots of $f$ and its iterates are simple (see \cite[Lemma 2.2]{Ber19}). If Lubin's conjecture holds, $\Log_f$ is the logarithm of the latent formal group. In any case, the formal power series
        \[
            F_f(X, Y) = \Log_f^{\circ -1}(\Log_f(X) + \Log_f(Y)) \in K[\![X, Y]\!]
        \]
        is the unique formal group over $K$ for which both $f$ and $u$ are endomorphisms. We show in Subsection \ref{4.6. The associated formal group} that $F_f$ is in fact defined over $\mathcal{O}_K$. For $m \in \mathbb{Z}_p$, we denote by $[m]_f \in \mathscr{S}_0(K)$ the unique formal power series over $K$ commuting with $f$ whose linear coefficient is $m$.

        \subsubsection{Torsion points} The set of roots of $f$ and its iterates is denoted by
        \[
                \Lambda(f)
            =   \{x \in \mathfrak{m}_{\mathbb{C}_p\!} \mid \exists\, n \geqslant 0,\, f^{\circ n}(x) = 0\}
            =   \bigsqcup_{n \geqslant 0} \Lambda_n(f)
        \]
        where $\Lambda_0(f) = \{0\}$ and, for all $n \geqslant 1$, $\Lambda_n(f)$ is the set of roots of $f^{\circ n}$ that are not roots of $f^{\circ n-1}$. Lubin proved in \cite[Proposition 3.2]{Lub94} that $\Lambda(f)$ is exactly the set of periodic points of $u$. We proceed to compute the valuations of these points, be means of Newton polygons. To plot such diagrams, we use the $p$-adic valuation $\upsilon_p$ so that $\upsilon_p(p) = 1$.

        \begin{proposition}
            For every integer $n \geqslant 1$, the endpoints of the decreasing part of the Newton polygon of $f^{\circ n}\!$ are the $(p^i, n-i)$ for all integers $i \in \{0,\dots, n\}$.
        \end{proposition}

        \begin{proof}
            The proof is by induction on $n$. The base case follows from the hypotheses on $f$ and \cite[Theorem 4.1]{Li96a}. Assume now that the result holds for $n \geqslant 1$. Since the linear coefficient and the Weierstrass degree are multiplicative, $(1, n+1)$ and $(p^{n+1}, 0)$ are vertices of the Newton polygon of $f^{\circ n+1}$. The solutions of the equations $f^{\circ n}(T) - x = 0$ where $x \in \Lambda_1(f)$ exhaust $\Lambda_{n+1}(f)$. Given $x \in \Lambda_1(f)$, the Newton polygon of $f^{\circ n} - x$ has a single line of slope $-1/(p^{n+1}-p^n)$ and of width $p^n$ by induction. Therefore all the elements of $\Lambda_{n+1}(f)$ have valuation $1/(p^{n+1}-p^n)$ and there are $p^{n+1} - p^n$ of them. This valuation is smaller than all those of the roots of $f^{\circ n}$, so the associated segment is the last one of the Newton polygon of $f^{\circ n+1}$. Its left-hand side part comes from the induction hypothesis.
        \end{proof}

    \subsection{The latent Tate module}\label{4.2. The latent Tate module} We now proceed to recover the missing structure on the latent $p$-adic Tate module. Given that $f'(0)$ is a uniformizer in $\mathbb{Z}_p$, we may take the set $T_f$ of $f$-consistent sequences of elements in $\mathfrak{m}_{\mathbb{C}_p\!}$ for the underlying set of the $p$-adic Tate module:
    \[
        T_f = \big\{ (x_n)_{n \geqslant 0} \in \mathfrak{m}_{\mathbb{C}_p\!}^\mathbb{N} \;\big|\; x_0 = 0 \text{ and } \forall\, n \geqslant 0,\, f(x_{n+1}) = x_n \big\}.
    \]
    Since $f$ is defined over $\mathcal{O}_K$, the group $G_K$ acts on $T_f$ and makes it a $G_K$-set. With respect to this structure, any power series in $\mathscr{S}_0(\mathcal{O}_K)$ that commutes with $f$ is an endomorphism, acting by componentwise evaluation. To generalize the recovering of the attached character (and later of the full structure of the Tate module) from \cite[\S2]{De26a} via the action of $u$, we use as starting point \cite[Corollary 4.1.2]{Li96a} for which the Weierstrass degree of $u(T) - T$ needs to be sufficiently large. This may be assumed: since $u'(0) \in \mathbb{Z}_p^\times$ is nontorsion, we may assume, up to replacing $u$ by one of its $\mathbb{Z}$-iterates, that
    \[
        u'(0) \in 1+p\mathbb{Z}_p.
    \]
    The group $u^{\circ \mathbb{Z}_p}$ is a pro-$p$-group acting on $\Lambda_1(f)$ having cardinality $p-1$, therefore $\wideg(u(T) - T) \geqslant p$ which implies that the sequence $(\wideg(u^{\circ p^n}(T) - T))_{n \geqslant 0}$ is strictly increasing, so that we may pick a $\mathbb{Z}$-iterate of $u$ satisfying the condition.

        \subsubsection{Recovering a Galois character}\label{4.2.1. Recovering a Galois character} Let $S_f \subseteq T_f$ be the subset of elements whose second entry is nonzero. The $n$-th coordinate of an element of $S_f$ is in $\Lambda_n(f)$ and, for every $n \geqslant 0$ and $x \in \Lambda_n(f)$, there exists an element of $S_f$ whose $n$-th coordinate is $x$. Furthermore, if Lubin's conjecture holds, $T_f$ is canonically isomorphic to the $p$-adic Tate module of the associated formal group, therefore is a free $\mathbb{Z}_p$-module of rank $1$, and an element $x \in T_f$ is a generator as $\mathbb{Z}_p$-module if and only if $x \in S_f$. Nonzero elements of $T_f$ are concatenation of a zero tuple with an element of $S_f$. In order to describe the action of $G_K$ on $T_f$, it is therefore sufficient to describe it on the $G_K$-subset $S_f$. Let
        \[
            \ell = \upsilon_p(u'(0) - 1) \geqslant 1.
        \]
        By  \cite[Corollary 4.1.2]{Li96a}, the fixed points of $u$ are the elements of $\bigsqcup_{n = 0}^\ell \Lambda_n(f)$, and hence $u$ only fixes the $\ell + 1$ first coordinates of the elements of $S_f$. As stated in \S\ref{2.1. Motivation}, from $u$, we will not be able to recover the action of $G_K$ but only the action of $G_L$ for some finite extension $L/K$. We interpret $S_f$ as a directed tree, as shown in Figure \ref{Figure 2}.

        \begin{figure}[h]
            \begin{center}
                \begin{tikzcd}[column sep=-2.25mm, row sep=4.4mm, arrows=-]
                    & \vdots \ar[dr] & \vdots \ar[d] & \vdots \ar[dl] & \vdots \ar[dr] & \vdots \ar[d] & \vdots \ar[dl] & \vdots \ar[dr] & \vdots \ar[d] & \vdots \ar[dl] &
                    \vdots \ar[dr] & \vdots \ar[d] & \vdots \ar[dl] & \vdots \ar[dr] & \vdots \ar[d] & \vdots \ar[dl] & \vdots \ar[dr] & \vdots \ar[d] & \vdots \ar[dl] &
                    \vdots \ar[dr, very thick] & \vdots \ar[d, very thick] & \vdots \ar[dl, very thick] & \vdots \ar[dr, very thick] & \vdots \ar[d, very thick] & \vdots \ar[dl, very thick] & \vdots \ar[dr, very thick] & \vdots \ar[d, very thick] & \vdots \ar[dl, very thick] &&
                    \vdots \ar[dr] & \vdots \ar[d] & \vdots \ar[dl] & \vdots \ar[dr] & \vdots \ar[d] & \vdots \ar[dl] & \vdots \ar[dr] & \vdots \ar[d] & \vdots \ar[dl] &
                    \vdots \ar[dr] & \vdots \ar[d] & \vdots \ar[dl] & \vdots \ar[dr] & \vdots \ar[d] & \vdots \ar[dl] & \vdots \ar[dr] & \vdots \ar[d] & \vdots \ar[dl] &
                    \vdots \ar[dr] & \vdots \ar[d] & \vdots \ar[dl] & \vdots \ar[dr] & \vdots \ar[d] & \vdots \ar[dl] & \vdots \ar[dr] & \vdots \ar[d] & \vdots \ar[dl] &
                    \\
                    && \ast \ar[drrr] &&& \ast \ar[d] &&& \ast \ar[dlll] &&& \ast \ar[drrr] &&& \ast \ar[d] &&& \ast \ar[dlll] &&& \ast \ar[drrr, very thick] &&& \ast \ar[d, very thick] &&& \ast \ar[dlll, very thick] &&&& \ast \ar[drrr] &&& \ast \ar[d] &&& \ast \ar[dlll] &&& \ast \ar[drrr] &&& \ast \ar[d] &&& \ast \ar[dlll] &&& \ast \ar[drrr] &&& \ast \ar[d] &&& \ast \ar[dlll]
                    \\
                    \ell \ar[rrrrr, dashed] &&&&& \ast \ar[drrrrrrrrr]\ar[rrrrrrrrr, dashed] &&&&&&&&& \ast \ar[d]\ar[rrrrrrrrr, dashed] &&&&&&&&& \ast \ar[dlllllllll, very thick]\ar[rrrrrrrrrr, dashed] &&&&&&&&&& \ast \ar[drrrrrrrrr]\ar[rrrrrrrrr, dashed] &&&&&&&&& \ast \ar[d]\ar[rrrrrrrrr, dashed] &&&&&&&&& \ast \ar[dlllllllll]\ar[rrrrr, dashed] &&&&& \phantom{\ell}
                    \\
                    &&&&&&&&&&&&&& \ast \ar[drrrrrrrrrrrrrr, very thick] &&&&&&&&&&&&&&&&&&&&&&&&&&&& \ast \ar[dllllllllllllll]
                    \\
                    &&&&&&&&&&&&&&&&&&&&&&&&&&&& 0
                \end{tikzcd}
            \end{center}
            \caption{The arboreal interpretation of $S_f$ when $p = 3$ and $\ell = 2$.}
            \label{Figure 2}
        \end{figure}

        \noindent Indexed from bottom to top, the vertices at level $n \geqslant 0$ are the elements of $\Lambda_n(f)$. For the sake of readability, we have omitted the arrowheads in the figure. There is an edge from a vertex $v_2$ to a vertex $v_1$ if and only if $f(v_2) = v_1$. The bold subpart represents the $u^{\circ \mathbb{Z}_p}$-orbit of any sequence inside this subpart, as we shall prove now.

        \begin{lemma}\label{R: u-orbits}
            For all element $\pi = (\pi_n)_{n \geqslant 0} \in S_f$ and for every integer $n \geqslant 1$, one has that $\{u^{\circ m}(\pi_{\ell + n}) \;;\; m \in \mathbb{Z}_p\} = \{x \in \mathfrak{m}_{\mathbb{C}_p}\! \mid f^{\circ n}(x) = \pi_\ell\}$.
        \end{lemma}

        \begin{proof}
            We proceed by induction on $n$. Let $(\pi_n)_{n \geqslant 0} \in S_f$. Since the power series $f$ and $u$ commute, it follows that
            \[
                \{u^{\circ m}(\pi_{\ell + 1}) \;;\; m \in \mathbb{Z}_p\} \subseteq \{x \in \mathfrak{m}_{\mathbb{C}_p}\! \mid f(x) = \pi_\ell\}.
            \]
            The RHS set has cardinality $p$. The LHS one has cardinality a power of $p$ (see the discussion on \cite[p. 333]{Lub94}) and cannot be $1$ because $u(\pi_{\ell + 1}) \neq \pi_{\ell + 1}$. These two sets are thus equals. Assume now that the result holds for $n \geqslant 1$ and for all $\pi \in S_f$. This means that $u^{\circ \mathbb{Z}_p}$ acts transitively on the roots of $f^{\circ n} - x$ for every $x \in \Lambda_\ell(f)$. Because $u$ is invertible and these sets have cardinality $p^n$, one must have that the power series $u^{\circ p^n}$ acts trivially on them, hence on $\Lambda_{\ell + n}(f)$, and therefore on all the roots of $f^{\circ \ell + n}$. The Newton polygon of $u^{\circ p^n}(T) - T$ starts at $(1, \ell + n)$ and have a segment for each of $\Lambda_1(f), \dots, \Lambda_{\ell + n}(f)$. This implies that $u^{\circ p^n}$ can have no other fixed points. Let $(\pi_n)_{n \geqslant 0} \in S_f$, one has the inclusion
            \[
                \{u^{\circ m}(\pi_{\ell + n + 1}) \;;\; m \in \mathbb{Z}_p\} \subseteq \{x \in \mathfrak{m}_{\mathbb{C}_p}\! \mid f^{\circ n + 1} (x) = \pi_\ell\}.
            \]
            The second set has cardinality $p^{n+1}$. The first one has cardinality a power of $p$ and at least $p^n$ by induction. If it were $p^n$, then $\pi_{\ell + n +1} \in \Lambda_{\ell + n + 1}(f)$ would be a fixed point of $u^{\circ p^n}$ which cannot happen. We deduce the equality.
        \end{proof}

        \begin{remark}
            Let $\pi = (\pi_n)_{n \geqslant 0} \in S_f$ and set $K_n = K(\pi_n)$ for all $n \geqslant 0$. This defines a tower of extensions of $K$. The previous proof shows that the extension $K_{\ell + n}/K_\ell$ is Galois for all $n \geqslant 0$. Indeed, $\pi_{\ell + n}$ is a root of the Weierstrass polynomial attached to $f^{\circ n} - \pi_\ell$ which is defined over $K_\ell$ and we proved that all the other roots of this polynomial are of the form $u^{\circ m}(\pi_{\ell+n})$ which lies in $K_{\ell + n}$. In particular $K_{\ell+n}/K_\ell$ is cyclic, and so $K(\pi)/K_\ell$ is procyclic.
        \end{remark}

        For a $\pi = (\pi_n)_{n \geqslant 0} \in S_f$, we define
        \[
            S_{f, \pi} = \{(x_n)_{n \geqslant 0} \in T_f \mid x_0 = \pi_0,\dots, x_\ell = \pi_\ell\}
        \]
        which is a subset of $S_f$. Setting $j = p^\ell - p^{\ell - 1}$; any family $(\alpha_1, \dots, \alpha_j)$ of elements $\alpha_i = (\alpha_{i,n})_{n \geqslant 0}$ of $S_f$ such that $\{\alpha_{1,\ell},\dots,\alpha_{j,\ell}\} = \Lambda_\ell(f)$ then determines a partition $\{S_{f, \alpha_1},\dots, S_{f, \alpha_j}\}$ of the set $S_f$. We now proceed to retrieve the Galois action on each of the cells of such a partition.

        \begin{proposition}\label{R: action-u}
            For all $\pi \in S_f$, the group $u^{\circ \mathbb{Z}_p}\!$ acts simply transitively on $S_{f, \pi}$.
        \end{proposition}

        \begin{proof}
            The group of $\mathbb{Z}_p$-iterates of $u$ acts transitively on the set of $n$-th coordinate of the elements of $S_{f, \pi}$, for all $n \geqslant 0$, by Lemma \ref{R: u-orbits}. Transitivity on the limit $S_f$ is guaranteed by compacity of $\mathbb{Z}_p$. Freeness follows from \cite[Corollary 4.3.1]{Lub94}.
        \end{proof}

        Let $L = K(\Lambda_\ell(f))$ be the finite extension of $K$ generated by the roots of $f^{\circ \ell}$ or, equivalently, by the fixed points of $u$. The previous proposition allows to describe the action of $G_L$ on $T_f$ via the dynamics of $u$.  Let $\pi = (\pi_n)_{n \geqslant 0} \in S_f$. For $g \in G_L$, the element $g.\pi$ lies in $S_{f, \pi}$ and there exists a unique $\mathbb{Z}_p$-iterate $u_g$ of $u$ such that $g.\pi = u_g(\pi)$ according to Proposition \ref{R: action-u}. This association defines a map
        \[
            \chi_{f, \pi} \colon G_L \longrightarrow \mathbb{Z}_p^\times \colon g \longmapsto u_g'(0)
        \]
        which is a character, as we state now.

        \begin{proposition}
            For all $\pi \in S_f$, the associated map $\chi_{f, \pi} \colon G_L \rightarrow \mathbb{Z}_p^\times$ is a character satisfying $g.\pi = [\chi_{f, \pi}(g)]_f(\pi)$ for all $g \in G_L$.
        \end{proposition}

        \begin{proof}
            The fact that $\chi_{f, \pi}$ is a character is a simple consequence of the freeness of the action in Proposition \ref{R: action-u}. The rest of the statement is by construction.
        \end{proof}

        \subsubsection{Regularity of the character}\label{4.2.2. Regularity of the character} For the remainder of this subsection, we choose an $f$-consistent sequence $\pi = (\pi_n)_{n \geqslant 0} \in S_f$ and simply write
        \[
            \chi_f \colon G_L \longrightarrow \mathbb{Z}_p^\times \colon g \longmapsto \chi_{f, \pi}(g)
        \]
        its character. After Theorem \ref{main-theorem}, it will follow that the subsequent constructions were independent of the choice of $\pi$. In this paragraph, we proceed to show that $\chi_f$ is crystalline of weight $1$ (see \S\ref{3.2. Galois p-adic representations} for a definition) as in \cite[\S 2.4]{De26a}. We repeat the constructions from \textit{ibidem} as we will discuss their functoriality in Section \ref{5. Functoriality}. The key point lies in the construction of a ring which is an adaptation of the ring $A_\infty$ in \cite{Spe18} for not necessarily Lubin-Tate formal groups. Let $\pi_L$ be a uniformizer of the ring $\mathcal{O}_L$, let $r$ be the residue degree of the extension $L/\mathbb{Q}_p$ and write $q = p^r$. By \cite[Corollary 6.2.1]{Lub94}, there exists an invertible formal power series $a \in \mathscr{S}_0(\mathbb{F}_q)$ such that
        \[
            f^{\circ r}(T) \bmod (\pi_L) = a(T^q).
        \]
        Let $A_\infty(\mathcal{O}_L)$ be the $\pi_L$-adic completion of the ring $\varinjlim \mathcal{O}_L[\![X_n]\!]$ with the transition maps $X_n \mapsto f^{\circ r}(X_{n+1})$ for all integers $n \geqslant 0$. We endow it with an action of $G_L$~by letting $g.X_n = [\chi_f(g)]_f(X_n)$ for all $g \in G_L$. This action is well defined because the formal power series $[\chi_f(g)]_f$ are $\mathbb{Z}_p$-iterates of $u$. Under the $(\pi_L, X_0)$-adic topology, the $X_n$ are topologically nilpotent elements in $A_\infty(\mathcal{O}_L)$.

        \begin{lemma}
            The ring $A_\infty(\mathcal{O}_L)$ is a strict $\pi_L$-ring.
        \end{lemma}

        \begin{proof}
            A ring $R$ is a strict $\pi_L$-ring if it is $\pi_L$-adically complete, $\pi_L$-torsion free and the residue ring $R/(\pi_L)$ is perfect. In our situation, only the third condition needs to be checked. For this, observe that
            \begin{align*}
                    A_\infty(\mathcal{O}_L)/(\pi_L)
                &   \simeq \varinjlim_{X_n \mapsto a(X_{n+1}^q)} \mathcal{O}_L[\![X_n]\!]/(\pi_L)\\
                &   \simeq \varinjlim_{X_n \mapsto a(X_{n+1})^q} \mathbb{F}_q[\![X_n]\!]\\
                &   \simeq \varinjlim_{a^{\circ n}(X_n) \mapsto a^{\circ n+1}(X_{n+1})^q} \mathbb{F}_q[\![a^{\circ n}(X_n)]\!]\\
                &   \simeq \varinjlim_{Y_n \mapsto Y_{n+1}^q} \mathbb{F}_q[\![Y_n]\!]
            \end{align*}
            which is a perfect ring, where we let $Y_n = a^{\circ n}(X_n)$ for all $n \geqslant 0$.
        \end{proof}

        In order to construct a $p$-adic period of our character $\chi_f$, we first pass from $T_f$ to the ring $\mathcal{R}$ (defined in \S\ref{3.1. Rings for p-adic Hodge theory}) via the map
        \[
            \vartheta_f \colon T_f \longrightarrow \mathcal{R} \colon (y_n)_{n \geqslant 0} \longmapsto (a^{\circ n}(y_{nr} \bmod \pi_L))_{n \geqslant 0}.
        \]
        Then we apply the universal property of the ring $A_\infty(\mathcal{O}_L)$ to the pair $(\mathcal{R}, \varpi)$, where $\varpi = (\varpi_n)_{n \geqslant 0}$ is the $f^{\circ r}$-consistent sequence of topologically nilpotent elements in $\mathcal{R}$ given by
        \[
            \varpi_n = (a^{\circ m}(\pi_{(n+m)r} \bmod \pi_L))_{m \geqslant 0}
        \]
        for every $n \geqslant 0$. Let $\phi_\varpi \colon A_\infty(\mathcal{O}_L) \rightarrow \mathcal{R}$ be the homomorphism determined by the requirement that $\phi_\varpi(X_n) = \varpi_n$ for all $n \geqslant 0$. As $\pi_L$ is zero in $\mathcal{R}$, this map factors through $(\pi_L)$, and we proceed to show that it is its kernel.

        \begin{lemma}
            The homomorphism $\phi_\varpi$ is $G_L$-equivariant with kernel $\pi_L A_\infty(\mathcal{O}_L)$.
        \end{lemma}

        \begin{proof}
            The first assertion follows from the fact that all the power series $[\chi_f(-)]_f$ commutes with $f$. The second is the same as in \cite[Lemma 2.2]{Ber16}: if the induced map $\mathbb{F}_q[\![X_n]\!] \rightarrow \mathcal{R}$ were not injective, there would exist a nonzero $P(X_n) \in \mathbb{F}_q[X_n]$ such that $P(\overline{\phi_\varpi}(X_n)) = 0$ and this would imply that $\varpi_n$ is algebraic over $\mathbb{F}_p$, which is not the case.
        \end{proof}

        Since $\AinfL = \W_{\!\pi_L\!}(\mathcal{R})$ is also a strict $\pi_L$-ring, the factorization of $\phi_\varpi$ through its kernel lifts to a $G_L$-equivariant embedding $A_\infty(\mathcal{O}_L) \hookrightarrow \AinfL$ via:
        \[
            \begin{tikzcd}[row sep=11mm, column sep=10mm]
                & \mathcal{R} & \AinfL \ar[l, twoheadrightarrow]\\
                A_\infty(\mathcal{O}_L) \ar[ur, "\phi_\varpi"]\ar[r, twoheadrightarrow] & A_\infty(\mathcal{O}_L)/(\pi_L) \ar[u, hookrightarrow] & A_\infty(\mathcal{O}_L). \ar[l, twoheadrightarrow]\ar[u, hookrightarrow]
            \end{tikzcd}
        \]
        Under this embedding, we let $x_n$ be the image of $X_n$ for every $n \geqslant 0$. Equivariance of the map implies that $g.x_n = [\chi_f(g)]_f(x_n)$ for all $g \in G_L$. Let
        \[
            t_f = \Log_f(x_0) \in \BdR^+.
        \]
        This element lies in $\BdR^+$ according to \cite[Lemma 3.2]{Ber14} since $\Log_f$ converges on the open unit disk of $\mathbb{C}_p$ and $\theta(x_n) \equiv \pi_{nr} \bmod (\pi_L)$. Notice that picking $\Log_f(x_n)$ for another $n > 0$ would yield a $\mathbb{Q}_p$-multiple of $t_f$. Existence of $t_f$ guarantees that $\chi_f$ is crystalline of weight $1$, as we prove now.

        \begin{proposition}
            The character $\chi_f\!$ of $G_L\!$ is crystalline of Hodge-Tate weight $1$.
        \end{proposition}

        \begin{proof}
            This is the same as \cite[Proposition 2.6]{De26a}. We repeat the proof due to the stature of the result. The latter convergence in $\BdR^+$ is with respect to the topology generated by the ideal $\ker(\theta)$. So there exists some $m \geqslant 1$ for which
            \[
                \theta(b_mx_0^m) = b_m \theta(x_0)^m = 0,
            \]
            implying that $\theta(x_0) = 0$, where $b_m \in K^\times$ is the $m$-th coefficient of the formal power series $\Log_f$. Therefore $\theta(t_f) = 0$, and hence $\chi_f$ has weight at least $1$. If it were $2$, the $\upsilon_\mathcal{R}$-valuation of $x_0 \bmod (\pi_L) = \varpi_0$ would be at least $2$, however
            \begin{align*}
                \upsilon_\mathcal{R}(\varpi_0)
                & = \lim_{m \to +\infty} q^m\upsilon_p(\pi_{mr})\\
                & = \lim_{m \to +\infty} p^{mr}\frac{1}{p^{mr-1}(p-1)}\\
                & = \frac{p}{p-1}
            \end{align*}
            is less than $2$ when $p$ is odd; a contradiction. For $p = 2$, it suffices to compute the $\upsilon_\mathcal{R}$-valuation of the next component of $x_0$ and this leads to a contradiction too. In particular $t_f$ is nonzero. For every $g \in G_L$,
            \[
                g.t_f = \Log_f(g.x_0) = \Log_f([\chi_f(g)]_f(x_0)) = \chi_f(g)t_f
            \]
            and hence $\chi_f$ is de Rham of Hodge-Tate weight $1$. To show that $\chi_f$ is crystalline, first observe that $t_f \in L \otimes_{L_0} \Bcris^+$ according to \cite[\S 4.1.3]{Fon94} since $\theta(x_0) = 0$ and $\Log_f'(T) \in \mathcal{O}_K[\![T]\!]$. This guarantees the existence of a crystalline period because
            \begin{align*}
                \dim_{\mathbb{Q}_p} \mathbb{Q}_p(\chi_f)
                & = \dim_{L} \Hom_{\mathbb{Q}_p[G_L]}(\mathbb{Q}_p(\chi_f), L \otimes_{L_0} \Bcris)\\
                & = \dim_{L} L \otimes_{L_0} \Hom_{\mathbb{Q}_p[G_L]}(\mathbb{Q}_p(\chi_f), \Bcris)\\
                & = \dim_{L_0} \D^*_\mathrm{cris}(\mathbb{Q}_p(\chi_f)),
            \end{align*}
            which concludes the proof.
        \end{proof}

        \subsubsection{Reconstructing the Tate module} Let $F$ be a height-one formal group defined over $\mathcal{O}_L$ whose $p$-adic Tate module is isomorphic to $\mathbb{Z}_p(\chi_f)$; see paragraph \ref{3.3.2. Crystalline representations} for its existence. Since they are isomorphic free $\mathbb{Z}_p[G_L]$-modules of rank $1$, they share the same character. We proceed to transport the structure of $\T_p(F)$ to $T_f$ in order to endow the latter with a $\mathbb{Z}_p$-module structure. To do so, we start by identifying the unique endomorphism of $F$ to which $f$ should correspond. This can be done by examining the action of $f$ on $t_f$, which is thus a shared period. More generally, we let $s \in \mathscr{S}_0(\mathcal{O}_K)$ be a formal power series commuting with $f$, and we want to define $s(t_f)$ arising from the constructions in \S\ref{4.2.2. Regularity of the character}. Recall that we have established the following path in order to obtain a period:
        \[
            \begin{tikzcd}[row sep=2mm, column sep=5mm]
                & \Aty(\mathcal{O}_L) \ar[rr, hookrightarrow] && \AinfL \ar[rr, dashed, "\Log_f"] && L \otimes_{L_0} \Bcris^+\\
                \\
                & X_0 \ar[rr, mapsto]\ar[dddr, mapsto] && x_0 \ar[rr, mapsto]\ar[dddl, mapsto] && t_f\\
                T_f && \mathcal{R}\ar[from=uuur, twoheadrightarrow, crossing over]\ar[from=uuul, crossing over, "\phi_\varpi"] \ar[from=ll, crossing over, "\vartheta_f"]\\
                \\
                \pi \ar[rr, mapsto] && \varpi_0
            \end{tikzcd}
        \]
        where the dashed arrow is defined for all $x \in \AinfL \cap \ker(\theta)$. Considering maps in the same order as we did, $s(\pi)$ is sent to $(a^{\circ m}(s(\pi_{mr}) \bmod \pi_L))_{m \geqslant 0}$ which is well defined in $\mathcal{R}$ because $s \bmod (\pi_L)$ commutes with $a$. This element is also $s(\varpi_0)$ and is the image of $s(X_0)$ under $\phi_\varpi$. Therefore these maps agree after action of $s$. By functoriality of the Witt vectors functor, $s(X_0)$ is sent to $s(x_0)$ and is a lift of $s(\varpi_0)$ modulo $(\pi_L)$. As $s(x_0) \in \ker(\theta)$, we let
        \[
            s(t_f) = \Log_f(s(x_0)) = s'(0)t_f.
        \]
        From the perspective of the formal group $F$, such an action on $t_f$ arises from $s_F$, the multiplication-by-$s'(0)$ endomorphism of $F$. Indeed, the functor from $\mathbf{FGL}/\mathcal{O}_L$ to $\MF_L(\varphi)$ of Section \ref{3. Theoretical background} is faithful, and the power series $s_F$ acts in the same way on all de Rham periods of $\T_p(F)$, including the crystalline ones. We must therefore associate $f$ with $f_F$ when defining a bijection between $T_f$ and $\T_p(F)$.

        \begin{proposition}\label{R: p-adic-Tate-module}
            There exists a $\mathbb{Z}_p[G_L]$-module structure on the set $T_f$ for which it is a crystalline $G_L$-character of Hodge-Tate weight $1$ and $f \in \End_{\mathbb{Z}_p[G_L]}(T_f)$.
        \end{proposition}

        \begin{proof}
            Notice that since $F$ has height-one, its group $F[p^n]$ of $p^n$-torsion points has cardinality $p^n$, and these points are also the roots of $f^{\circ n}_F$ because the series $f_F$ and $[p]_F$ commute and $f'(0)$ is a uniformizer in $\mathbb{Z}_p$. To easily define a bijection from $T_f$ to $\T_p(F)$, we reinterpret the latter by:
            \[
                \T_p(F) = \varprojlim_{x \mapsto [p]_F(x)} F[p^n] \simeq \varprojlim_{x \mapsto f_F(x)} F[p^n].
            \]
            The same manner as for the constructions in \S\ref{4.2.1. Recovering a Galois character}, we consider $S_F$ the subset of $\T_p(F)$ of elements with a nonzero second entry. Given $\beta = (\beta_n)_{n \geqslant 0} \in S_F$, let
            \[
                S_{F, \beta} = \{(y_n)_{n \geqslant 0} \in \T_p(F) \mid y_0 = \beta_0, \dots, y_\ell = \beta_\ell\} \subset S_F
            \]
            be a cell. In order to define a $G_L$-equivariant bijection $\tau \colon T_f \rightarrow \T_p(F)$ transporting $f$ to $f_F$, we begin by matching $S_f$ with $S_F$. Set $j = p^\ell - p^{\ell - 1}$ and choose a system $(\alpha_1, \dots, \alpha_j)$ of elements of $S_f$ inducing a partition of $S_f$ as in \S\ref{4.2.1. Recovering a Galois character}. Similarly, pick a system $(\beta_1, \dots, \beta_j)$ of elements of $S_F$ where we run through the torsion points of $F$ of exact order $p^\ell$. Such a system, in turn, induces a partition of $S_F$. For every integers $i \in \{1,\dots, j\}$ and $n \geqslant 0$, and for all $g \in G_L$, set
            \begin{equation}\label{eqn: 1}
                \tau((f^{\circ n} \circ [\chi_{f, \alpha_i}(g)]_f)(\alpha_i)) = (f^{\circ n}_F \circ [\chi_f(g)]_F)(\beta_i)
            \end{equation}
            and $\tau(0) = 0$. Fixing $n = 0$ and taking $i \in \{1,\dots, j\}$ in (\ref{eqn: 1}) furnishes a well-defined bijection $\tau \colon S_{f, \alpha_i} \rightarrow S_{F, \beta_i}$ since the $\mathbb{Z}_p$-iterates of $u$ act simply transitively on $S_{f, \alpha_i}$ and the same holds about $u_F$ and $S_{F, \beta_i}$. Varying $n \geqslant 0$ in (\ref{eqn: 1}) allows to reach the remaining points of $T_f$ and $\T_p(F)$. Therefore (\ref{eqn: 1}) defines a $G_L$-equivariant bijection $\tau \colon T_f \rightarrow \T_p(F)$ associating $f$ with $f_F$, and the statement follows by transport of structure.
        \end{proof}

        \begin{remark}
            Observe that a different choice of representatives for the partitions in the preceding proof would yield the same $\mathbb{Z}_p[G_L]$-module structure on $T_f$, up to a permutation of the roles of its elements. This is immaterial for our purpose since what matters to us are the group law and the property that $f$ is an endomorphism of $T_f$ after transport of structure; both are independent of these choices. Moreover, any other formal group over $\mathcal{O}_L$ whose $p$-adic Tate module is isomorphic to $\mathbb{Z}_p(\chi_f)$ would induce the same structure on $T_f$.
        \end{remark}

        Starting from the data $(T_f, f)$ endowed with the structure of Proposition \ref{R: p-adic-Tate-module}, we now reconstruct the latent formal group following the method in \cite{De26a} recalled in \S\ref{3.3. Objects arising from formal groups}. The dashed path in Figure \ref{Figure 1} is sufficiently explicit to keep track of $f$ along these constructions, and ultimately to show that the resulting formal group is the latent one. There is, however, a subtlety concerning the action of $u$. The action of $G_L$ is locally given by characters that are linked to the action of the $\mathbb{Z}_p$-iterates of the series $u$. At this stage, we do not know whether these characters are all equal (as they will turn out to do after Theorem \ref{main-theorem}). To preserve the $G_L$-action when transporting the structure, we therefore set aside the tracking of $u$. This does not affect the tracking of $f$. Recall that $f$ acts on $T_f$ by componentwise evaluation or, equivalently, by multiplication by $f'(0)$, and this series is uniquely determined by the action of $\mathbb{Z}_p$ on $T_f$.

    \subsection{The associated Breuil-Kisin module}\label{4.3. The associated Breuil-Kisin module} Let
    \[
        \mathfrak{M}_f = \mathbf{D}_\mathfrak{S}^*(T_f)
    \]
    be the rank-one formal Breuil-Kisin modules attached to $T_f$. We examine how the endomorphism $f$ is transported. To this end, set $\mathscr{M}_f = \Hom_{\mathbb{Z}_p[G_{L_\infty}]}(T_f, \mathfrak{S}^\mathrm{ur})$ the free $\mathfrak{S}$-module of rank $1$ appearing in \cite[Lemma 2.1.10]{Kis06}. Therefore the action of $f$ on $T_f$ extends to the latter to be the multiplication by $f'(0)$: indeed, given a generator $(\pi \mapsto e)$ of $\mathscr{M}_f$, one has $f(\pi \mapsto e) = f'(0)(\pi \mapsto e)$ because $f(\pi) = f'(0)\pi$. The formal power series $f$ remains uniquely determined by the action of $\mathbb{Z}_p \subset \mathfrak{S}$ on $\mathscr{M}_f$. Now, we may assume that $\mathfrak{M}_f \subseteq \mathscr{M}_f$ as $\mathfrak{S}$-modules (\cite[Lemma 2.1.10]{Kis06}), therefore the action of $f$ on $\mathscr{M}_f$ restricts to $\mathfrak{M}_f$ and hence the formal power series $f$ can be uniquely recovered from this section.

    \subsection{The associated window} Let
    \[
        M_f = S \otimes_{\varphi, \mathfrak{S}} \mathfrak{M}_f
    \]
    be the $S$-window attached to our formal Breuil-Kisin module. Let $(M_{f, n})_{n \geqslant 1}$ be~the corresponding system for the limit description. Upon passing to the quotient of $M_f$ with which $M_{f, n}$ is identified, we may assume that $M_{f, n}$ is a quotient of $M_f$. The action of $f$ on $\mathfrak{M}_f$ extends to an action of $f$ on $M_f$ to be the multiplication by its linear coefficient, and this action descends compatibly to the quotient $M_{f, n}$. Taking the $M_{f, n}$ as quotients of $M_f$ guarantees that it is still $f$ that acts on these objects, and $f$ can be uniquely recovered from the action of $\mathbb{Z}_p \subset S$ on them.

    \subsection{The associated \texorpdfstring{$p$}{p}-divisible group} Let $n \geqslant 1$ be an integer and write $R_n$ for the ring $\mathcal{O}_L/(p^n)$. We set $\mathscr{P}_{f, n} = \Dsp(M_{f, n}) = \W(R_n) \otimes_{\varkappa_n , S} M_{f, n}$, which is a $\W(R_n)$-module, where $\varkappa_n \colon S \rightarrow \W(R_n)$ is the map of \cite[p. 499]{Zin01}. The action of $f$ on $M_{f, n}$ extends to an action on $\mathscr{P}_{f, n}$ to be the multiplication by $f'(0)$. The formal power series $f$ remains uniquely recoverable from the action of $\mathbb{Z}_p \subseteq \W(R_n)$ on $\mathscr{P}_{f, n}$. Let
    \[
        \Gamma_f = \bt(\Dsp(M_f)) = \varprojlim_{m \geqslant 1} \bt(\mathscr{P}_{f, m})
    \]
    be the height-one connected $p$-divisible group over $\mathcal{O}_L$ attached to $M_f$. Denote by $\W_\mathrm{f}$ the functor of Witt vectors with finite support. One defines $\Gamma_{f, n} = \bt(\mathscr{P}_{f, n})$~as a quotient of the functor $\W_\mathrm{f}(-) \otimes_{\W(R_n)} \mathscr{P}_{f, n}$ by \cite[Theorem 4]{Zin02}. The action of $f$ on $\mathscr{P}_{f, n}$ extends to the previous functor, descends to all the quotients $\Gamma_{f, n}$ and is compatible with the transition maps. Therefore the action of $f$ passes to the limit $\Gamma_f$ and is, again, the multiplication by its linear coefficient. Thus the formal power series $f$ remains an endomorphism at this step and is uniquely determined by the structure of $\Gamma_f$ which is induced by the one on $T_f$.

    \subsection{The associated formal group}\label{4.6. The associated formal group} Following the dashed path in Figure \ref{Figure 1}, we have constructed a pair $(\Gamma_f, f)$ where $\Gamma_f$ is a height-one connected $p$-divisible group defined over $\mathcal{O}_L$ and $f$ is an endomorphism of $\Gamma_f$ and is the descent data. We may now state the main result.

    \begin{theorem}\label{main-theorem}
        There exists a formal group $F_f$ over $\mathcal{O}_K$ such that $f, u \in \End_{\mathcal{O}_K}(F)$.
    \end{theorem}

    \begin{proof}
        The group law on $\Gamma_f$ induces a comultiplication on its Hopf algebra, which is a map
        \[
            \mathcal{O}_L[\![T]\!] \longrightarrow \mathcal{O}_L[\![X, Y]\!] \colon T \longmapsto F_f(X, Y)
        \]
        where $F_f(X, Y)$ is a height-one formal group defined over $\mathcal{O}_L$. By construction, $f$ is an endomorphism of $\Gamma_f$ and therefore is an endomorphism of $F_f$. According to the uniqueness property in \S\ref{4.1.1. The logarithm}, the formal group $F_f$ is defined over $\mathcal{O}_L \cap K = \mathcal{O}_K$. The formal power series $u$ is defined over $\mathcal{O}_K$ and commutes with $f$, it is therefore necessarily an $\mathcal{O}_K$-endomorphism of $F_f$ as well.
    \end{proof}

\section{Functoriality}\label{5. Functoriality}

\noindent In Section \ref{4. Constructions}, we postponed the discussion on the functoriality of our constructions in order to carry out the recovering of the latent formal group. We now revisit this question and prove that the assignments sending a height-one commuting pair to its (restricted) representation and to its formal group are functorial. We let $K$ be a finite extension of $\mathbb{Q}_p$.

    \subsection{Height-one Lubin pairs}\label{5.1. Height-one Lubin pairs} For a formal power series $s \in \mathscr{S}_0(\mathcal{O}_K)$, we define the commutant $\Comm(s)$ of $s$ over $\mathcal{O}_K$ as the subset of $\mathscr{S}_0(\mathcal{O}_K)$ of elements that commute with $s$. We say that $s$ is stable if $s'(0)$ is nonzero nor a root of $1$. Lubin studied in \cite[\S1]{Lub94} the commutant of such series.

    \begin{definition}
        The category $\mathbf{LP}^1\!/\mathcal{O}_K$ of height-one Lubin pairs over $\mathcal{O}_K$ has as:
        \begin{enumerate}
            \item[$\triangleright$] objects the pairs $(f, u)$ of commuting series in $\mathscr{S}_0(\mathcal{O}_K)$ such that $f'(0)$ is a uniformizer in $\mathbb{Z}_p$, $u'(0) \in \mathbb{Z}_p^\times$ is nontorsion, the roots of $f$ and its iterates are simple and $\wideg(f) = p$.
            \item[$\triangleright$] nonzero morphisms $(f_1, u_1) \rightarrow (f_2, u_2)$ the power series $m \in \mathscr{S}_0(\mathcal{O}_K)$ with $m'(0) \neq 0$ such that for all $s_1 \in \Comm(f_1)$, there exists $s_2 \in \Comm(f_2)$ for which $m \circ s_1 = s_2 \circ m$; zero morphism the null power series.
        \end{enumerate}
    \end{definition}

    \begin{remark}
        The power series $s_2$ appearing in the definition satisfies $s_2'(0) = s_1'(0)$, and is therefore unique (see \cite[Proposition 1]{Lub94}).
    \end{remark}

    \begin{remark}
        The definition of a morphism may seem somewhat technical at first, but it arises from the following consideration: if $m \colon F \rightarrow G$ is a homomorphism of formal groups, $m \circ [\alpha]_F = [\alpha]_G \circ m$ for all $\alpha \in \mathbb{Z}_p$. Conversely, it suffices for $m$ to link a single pair of stable series to be a homomorphism. Our definition appears to be much stronger but has the advantage of not relying on the outcome of Lubin's conjecture, of being independent of the choice of a pair inside a commutant and of ensuring that composition is well defined. However, notice that the commutant of a height-one Lubin pair is isomorphic to $\mathbb{Z}_p$ via the Jacobian by Theorem \ref{main-theorem}, and hence the condition on morphisms is equivalent to requiring it on a single pair of stable series, as one would expect.
    \end{remark}

    \begin{remark}
        Notice that the endomorphisms of a Lubin pair are the formal power series that belong to the commutant of the pair. A morphism is an isomorphism if and only if its linear coefficient is invertible. Therefore two height-one Lubin pairs arising from the same commutant are isomorphic, with the identity map providing such an isomorphism. This does not, however, mean that the two pairs are equal as objects of the category. We keep in mind that two distinct objects of $\mathbf{LP}^1\!/\mathcal{O}_K$ may be associated with the same formal group over $\mathcal{O}_K$.
    \end{remark}

    \begin{lemma}\label{R: relation-morphism-logs}
        Let $m \colon (f_1, u_1) \rightarrow (f_2, u_2)$ be a morphism of height-one Lubin pairs with logarithms $\Log_1\!$ and $\Log_2$. If $m$ is nonzero, $m'(0)\Log_1(T) = \Log_2(m(T))$.
    \end{lemma}

    \begin{proof}
        Denote by $\tilde{f}_2$ the unique element in $\Comm(f_2)$ such that $m \circ f_1 = \tilde{f}_2 \circ m$. In particular the linear coefficients of $f_1$ and $\tilde{f}_2$ are equal. Therefore
        \begin{align*}
            (\Log_2 \circ\; m) \circ f_1 = \Log_2 \circ \tilde{f}_2 \circ m = f_1'(0) . (\Log_2 \circ\; m).
        \end{align*}
        By the uniqueness property of the logarithm, we deduce that $(\Log_2 \circ\; m)/m'(0)$ is equal to $\Log_1$, and thus the statement.
    \end{proof}

    \subsection{Equivalence of categories}\label{5.2. Equivalence of categories} Recall from  \S\ref{4.1.1. The logarithm} the existence of a logarithm attached to a height-one Lubin pair $(f, u)$ over $\mathcal{O}_K$; this logarithm determines the unique formal group $F_f$ over $K$ admitting $(f, u)$ as a pair of endomorphisms. We have shown that $F_f$ is defined $\mathcal{O}_K$. This defines an association:
    \begin{align*}
        F \colon \mathbf{LP}^1\!/\mathcal{O}_K \longrightarrow \mathbf{FGL}^1\!/\mathcal{O}_K \colon (f, u) & \longmapsto F_f(X,Y)\\
        m & \longmapsto m
    \end{align*}
    which is functorial as we prove below, where $\mathbf{FGL}^1\!/\mathcal{O}_K$ is the full subcategory of $\mathbf{FGL}/\mathcal{O}_K$ consisting of height-one formal groups over $\mathcal{O}_K$.

    \begin{corollary}\label{R: functor}
        The association $F$ is an equivalence of categories.
    \end{corollary}

    \begin{proof}
        To prove that this association is functorial, it only remains to prove that it is well defined on morphisms. We let $(f_1, u_1)$ and $(f_2, u_2)$ be two height-one Lubin pairs over $\mathcal{O}_K$ and we denote by $\Log_1$ and $\Log_2$ their respective logarithms and by $F_1$ and $F_2$ their formal groups. If $m \colon (f_1, u_1) \rightarrow (f_2, u_2)$ is a nonzero morphism, we have that
        \begin{align*}
                m(F_1(X,Y))
            & = m(\Log_1^{\circ -1}(\Log_1(X) + \Log_1(Y)))\\
            & = \Log_2^{\circ -1}(m'(0)(\Log_1(X) + \Log_1(Y)))\\
            & = \Log_2^{\circ -1}(m'(0)\Log_1(X) + m'(0)\Log_1(Y))\\
            & = \Log_2^{\circ -1}(\Log_2(m(X)) + m'(0)\Log_2(m(Y)))\\
            & = F_2(m(X), m(Y))
        \end{align*}
        according to Lemma \ref{R: relation-morphism-logs}. Therefore $F$ is a functor. It is an equivalence because a quasi-inverse is given by sending $F(X, Y)$ to $([p]_F, [1+p]_F)$, which is a height-one Lubin pair.
    \end{proof}

    Before applying successive functors in integral $p$-adic Hodge theory, we recover the latent Tate module in \S\ref{4.2. The latent Tate module}. This construction defines an association:
    \begin{align*}
        T \colon \mathbf{LP}^1\!/\mathcal{O}_K \longrightarrow \varinjlim_{L/K \text{ finite}}\REP^{\mathrm{FGL}}_{\mathbb{Z}_p}(G_L) \colon (f, u) & \longmapsto T_f \\[-2mm]
        m & \longmapsto T(m)
    \end{align*}
    with $T(m)$ defined as follows. Set $T(0) = 0$. For a nonzero $m \colon (f_1, u_1) \rightarrow (f_2, u_2)$, denote by $\tilde{f}_2 \in \Comm(f_2)$ the unique series such that $m \circ f_1 = \tilde{f}_2 \circ m$. Let $v \in \mathbb{Z}_p^\times$ such that $f_1'(0) = v f_2'(0)$. By Theorem \ref{main-theorem}, there is a unique $\mathcal{O}_K$-endomorphism $[v]_2$ of the formal group of $(f_2, u_2)$ whose linear coefficient is $v$. Finally, we define $\psi \colon T_{\tilde{f}_2} \rightarrow T_{f_2}$ by sending $(x_n)_{n \geqslant 0}$ to $([v]_2^{\circ n}(x_n))_{n \geqslant 0}$. The map $T(m)$ is obtained by completing the following diagram:
    \[
        \begin{tikzcd}[row sep=11mm, column sep=15mm]
            T_{f_1} \ar[r, dashed, "T(m)"]\ar[dr, "m"] & T_{f_2}\\
            & T_{\tilde{f}_2}. \ar[u, "\psi"]
        \end{tikzcd}
    \]

    \begin{corollary}
        The association $T$ is functorial.
    \end{corollary}

    \subsection{Example: full height-one commuting families}\label{SS: full-ht1} In \cite{De26b}, we studied full height-one commuting families of formal power series defined over $\mathcal{O}_K$ and we prove that such a family is the family of $\mathcal{O}_K$-endomorphisms of a height-one formal group over $\mathcal{O}_K$. This can be viewed as a corollary of our main theorem.

    \begin{corollary}
        Let $D = \{s_n \in \mathscr{S}_0(\mathcal{O}_K) \;;\; n \in \mathbb{Z}\}$ be a commuting family such that $s'_n(0) = n$ for all $n \in \mathbb{Z}$ and assume that $\wideg(s_p) = p$. Then there exists a formal group $F$ defined over $\mathcal{O}_K$ such that $D \subseteq \End_{\mathcal{O}_K}(F)$.
    \end{corollary}

    \begin{proof}
        From such a family $D$, we can define a height-one Lubin pair over $\mathcal{O}_K$ by taking $(s_p, s_{1+p})$. The only condition to check is that the roots of $s_p$ and its iterates are simple. This is \cite[Corollary 2.3]{De26b}, which follows from an elementary study of the dynamical system $D$. Then the statement follows from Corollary \ref{R: functor}.
    \end{proof}

\section*{Acknowledgments}

\noindent The author gratefully acknowledges the support and guidance of his supervisor Maja Volkov, whose valuable advice and encouragement greatly contributed to this work.

\end{document}